\documentclass[letterpaper, 10 pt, conference]{ieeeconf}  

\IEEEoverridecommandlockouts                              
\usepackage{amsmath}		
\usepackage{amscd}
\usepackage{amsfonts}
\usepackage{amssymb}		
\usepackage{mathtools}		

\usepackage[english]{babel}		
\usepackage[utf8]{inputenc}	
\usepackage{calc}
\usepackage{thumbpdf}			
\usepackage{latexsym}
\usepackage{xcolor}
\usepackage{enumerate}		
\usepackage{url}				
\usepackage{cite}	
\usepackage{graphics} 
\usepackage{epsfig} 
\usepackage{color}

\usepackage[bookmarks,bookmarksnumbered,colorlinks]{hyperref}
\hypersetup{linkcolor = black,anchorcolor = black,citecolor =
	black,filecolor = black,urlcolor = black}

\definecolor{darkgreen}{rgb}{0,0.6,0}

\newtheorem{definition}{Definition}
\newtheorem{lemma}{Lemma}
\newtheorem{thm}{Theorem}

\newtheorem{Bem}{Remark}

\def\eps{\varepsilon}
\newcommand{\norm}[1]{\left\lVert#1\right\rVert}
\newcommand{\RF}{K}
\newcommand{\R}{\mathbb{R}}
\newcommand{\N}{\mathbb{N}}

\newcommand{\K}{\mathcal{K}}
\newcommand{\U}{\mathcal{U}}
\newcommand{\Ltwo}[1]{L^2(\Omega,\mathcal{F}_k,\mathbb{P},#1)}
\newcommand{\Normal}{\mathcal{N}}
\newcommand{\F}{\mathcal{F}}
\newcommand{\Filtr}{(\F_k)_{k \in \N_0}}
\newcommand{\Prob}{\mathbb{P}}
\newcommand{\E}{\mathbb{E}}

\newcommand{\Exp}[1]{\mathbb{E}\left[#1\right]}

\newcommand{\xeq}{x^s}
\newcommand{\ueq}{u^s}
\newcommand{\Xstat}{X^s}
\newcommand{\Ustat}{U^s}
\newcommand{\Xstatopt}{X^s_*}
\newcommand{\Ustatopt}{U^s_*}

\title{\LARGE \bf
Pathwise turnpike and dissipativity results for discrete-time stochastic linear-quadratic optimal control problems
}

\author{Jonas Schießl, Ruchuan Ou, Timm Faulwasser, Michael Heinrich Baumann, and Lars Grüne
\thanks{
The authors gratefully acknowledge that this work was funded by the Deutsche Forschungsgemeinschaft (DFG, German Research Foundation) – project number 499435839.}
\thanks{Jonas Schießl, Michael Heinrich Baumann and Lars Grüne are with Mathematical Institute, University of Bayreuth, Germany,
        \tt\small \{jonas.schiessl,michael.baumann,lars.gruene\} @uni-bayreuth.de}%
\thanks{Timm Faulwasser and Ruchuan Ou are with Institute for Energy Systems, Energy Efficiency and Energy Economics, TU Dortmund University, Germany,
        \tt\small \{timm.faulwasser,ruchuan.ou\}@tu-dortmund.de}%
}

\begin{document}

\maketitle
\thispagestyle{empty}
\pagestyle{empty}

\begin{abstract}
    We investigate pathwise turnpike behavior of discrete-time stochastic linear-quadratic optimal control problems. Our analysis is based on a novel strict dissipativity notion for such problems, in which a stationary stochastic process replaces the optimal steady state of the deterministic setting.
    The analytical findings are illustrated by a numerical example.
\end{abstract}

\section{INTRODUCTION}

Turnpike properties are a valuable feature of optimal control problems (OCPs). 
In deterministic problems, the most common variant of this phenomenon describes that for increasing horizon lengths, the optimal trajectories spend most of their time near an optimal steady state. 
This means that the turnpike property can be seen as a finite-horizon variant of the usual asymptotic stability property. It was first observed by von Neumann and Ramsey, see \cite{Neumann1945, Ramsey1928}, but it is still a topic of recent research, particularly due to its relation to model predictive control, see \cite{Faulwasser2022, Gruene2017}.
A closely related concept is strict dissipativity. For deterministic problems the connection between strict dissipativity and the turnpike property is well studied, see \cite{Faulwasser2017, Gruene2016, Gruene2018}.

While turnpike phenomena are rather well understood in the deterministic setting, more theoretical work is needed when uncertainties enter the dynamics and the setting becomes stochastic. 
Some of the challenges here are the definition of the stochastic counterpart to the optimal steady state and the question of which objects we should work within the theoretical analysis, given that in numerical experiments one can observe turnpike properties concerning different objects like distributions, moments, or sample paths of the stochastic system, see \cite{Ou2021}.

The contribution of this paper is twofold: On the one hand we show a pathwise turnpike phenomenon in probability for stochastic systems in which the individual paths do not converge to a steady state but are subject to continued excitation by an additive disturbance. 
As we will see in Theorem~\ref{thm_TurnpikeProb}, this pathwise turnpike behavior states that there is a specific stationary process such that with high probability the paths of near-optimal processes stay close to corresponding paths of the stationary process, except for a number of time instances independent of the optimization horizon.
This distinguishes our results from existing ones, e.g., in \cite{Sun2022, Marimon1989, Kolokoltsov2012}, which either consider turnpike phenomena in distribution rather than pathwise or assume that the stochastic system is $L^2$-stabilizable, in contrast to our setting. On the other hand, we provide the---to the best of the authors' knowledge---first dissipativity-based analysis of a pathwise turnpike property for stochastic optimal control problems.

We investigate this connection in the discrete-time linear-quadratic setting. We note that our strict dissipativity notion differs from the one recently introduced in \cite{Gros2022}, because our notion cannot be reformulated in the sense of the underlying probability measures. In fact, a dissipativity notion in the sense of probability measures would not be strong enough to conclude the pathwise turnpike property in probability that we achieve in this paper, cf. Remark \ref{rem_StatProbMeasure} and Section \ref{sec_DissiAndTurnppike}.

The remainder of the paper is structured as follows: Section \ref{sec_setting} introduces the considered problem formulation and recalls the concepts of turnpike and dissipativity in the deterministic setting.
Section \ref{sec_stationary} shows that a pair of stationary stochastic processes fulfilling a suitable optimality criterion can replace the deterministic optimal steady state. Moreover, we show in Section \ref{sec_DissiAndTurnppike} that a time-varying dissipativity notion implies a pathwise turnpike property for this stationary pair. We also present an illustrative example in Section \ref{sec_example} and summarize our results in Section \ref{sec_summary}.

\section{SETTING AND PRELIMINARIES} \label{sec_setting}
Before investigating stochastic turnpike properties, we introduce our problem set-up and briefly recall the basic concepts of dissipativity and turnpike in the deterministic case.
\subsection{Problem formulation}
For a stabilizable pair $(A, B)$, we consider linear stochastic systems of the form 
\begin{equation} \label{eq_StochSystem}
    X(k+1) = AX(k) + BU(k) + W(k), \quad X(0) = X_0
\end{equation}
where at each time step $k \in \N_0$,  $X(k) \in \Ltwo{\R^n}$, $U(k) \in \Ltwo{\R^l}$, and $W(k) \in L^2(\Omega,\mathcal{F},\mathbb{P},\R^n)$. Here $\Omega$ is the set of realizations, $\Prob$ is the probability measure, $\F$ is a $\sigma$-algebra, and $\Filtr$ is a filtration following the usual hypotheses of \cite{Protter2005}.
For our purpose, we choose $\Filtr$ as the smallest filtration such that $X$ is an adapted process, i.e. 
\begin{equation*}
        \F_k = \sigma(X(0),\ldots,X(k)), \quad k \in \N_0.
\end{equation*}
This choice of the stochastic filtration induces a causality requirement, which ensures that the control action $U(k)$ at time $k$ only depends on the sequence of past disturbances $\lbrace W(k) \rbrace_{k=0,\ldots,k-1}$ and not on future events.
We refer to \cite{Fristedt1997, Protter2005} for more details on stochastic filtrations.\\
Further, for all $k \in \N_0$ we assume that $W(k) \sim \Normal(0, \Sigma_W)$ are \textit{i.i.d} Gaussian random variables which are independent of $X(k)$ and $U(k)$ and which have zero mean and covariance matrix $\Sigma_W \in \R^{n \times n}$.
For a given initial value $X_0 \sim \Normal(\mu_0,\Sigma_0)$ and control $U(\cdot)$ we denote the solution of system \eqref{eq_StochSystem} by $X_{U}(\cdot,X_0)$, or short by $X(\cdot)$ if the initial value and the control are unambiguous. 
Note, that the solution $X_{U}(\cdot,X_0)$ also depends on the disturbance $W(\cdot)$. However, for the sake of readability, we do not highlight this in our notation. \\
For two matrices $R \in \R^{l \times l}$, $Q \in \R^{n \times n}$ with $R > 0$, $Q \geq 0$ and $(A,Q^{1/2})$ detectable, the stage cost is given by 
\begin{equation} \label{eq_StageCost}
    \ell(X,U) := \Exp{\norm{X}_Q^2 + \norm{U}_R^2}.
\end{equation}
The stochastic optimal control problem (stochastic OCP) under consideration is
\begin{equation} \label{eq_sOCP}
    \min_{U(\cdot)} J_N(X_0,U) := \sum_{k=0}^{N-1} \ell(X(k),U(k))
\end{equation}
subject to \eqref{eq_StochSystem}.
The solution of problem \eqref{eq_sOCP} is well known, see \cite{CainesMayne1970}, and given by $U^*(k) = \RF_N(k) X_{U^*}(k,X_0)$ where 
\begin{equation*}
    \RF_N(k) := -\left[R + B^T P_N(k+1) B\right]^{-1} B^T P_N(k+1) A
\end{equation*} 
and $P_N(k)$ is the solution of the backward Riccati iteration (or Riccati difference equation)
\begin{equation} \label{eq_RiccatiIteration}
    \begin{split}
        P_N(k) =& A^T P_N(k+1) A + Q - A^T P_N(k+1)B \\ 
        &\times [ R + B^T P_N(k+1) B ]^{-1} B^T P_N(k+1)A
    \end{split}
\end{equation}
with the terminal condition $P_N(N) = 0$.

\subsection{Deterministic Dissipativity and Turnpike}
Let us denote the deterministic counterpart to system \eqref{eq_StochSystem} by
\begin{equation}
    x(k+1) = Ax(k) + Bu(k), \quad x(0) = x_0
\end{equation}
and the corresponding solution analogously by $x_{u}(\cdot,x_0)$. 
It is easy to see that $(\xeq,\ueq) = (0,0)$ solves the optimization problem 
\begin{equation} \label{eq_optStSt_det}
    \min_{x,u} \ell(x,u) = \norm{x}_Q^2 + \norm{u}_R^2 \quad \text{s.t.}~x=Ax+Bu.
\end{equation}
Therefore, we call $(\xeq,\ueq) = (0,0)$ the optimal steady state of the deterministic linear-quadratic optimal control problem. Next, we formalize the strict dissipativity property for the deterministic problem.
\begin{definition} \label{defn_DissiDeterm}
    The deterministic linear-quadratic optimal control problem is called \textit{strictly  $(x,u)$-dissipative} at $(\xeq,\ueq)$ if there exists a storage function $\lambda:~\R^n \rightarrow \R$ bounded from below and a function $\rho \in \K_{\infty}$, i.e., 
    $\rho: \R_0^+ \rightarrow \R_0^+$ continuous, strictly increasing and unbounded with $\rho(0)=0$,
    such that 
    \begin{equation}\label{eq:sDI}
        \begin{split}
            &\ell(x,u) - \ell(\xeq,\ueq) + \lambda(x) - \lambda(Ax+Bu) \\
            &\geq \rho(\norm{(x-\xeq,u-\ueq)})
        \end{split}
    \end{equation}
    holds for all $(x,u) \in \R^n \times \R^l$.  
\end{definition}

The following lemma, a slight modification of \cite[Theorem 5.3]{Gruene2013}, shows that strict dissipativity implies turnpike-like behavior for the deterministic case.
\begin{lemma} \label{lem_TurnpikeDeterm}
    Assume strict $(x,u)$-dissipativity at $(\xeq,\ueq)$. Then for each $x_0 \in \R^n$ there exists a constant $C \in \R$ such that for each $\delta > 0$, each control sequence $u(\cdot)$
    satisfying $J_N(x_0,u) \leq N\ell(\xeq,\ueq) + \delta$ and each $\eps > 0$ the value $Q_{\eps} := \# \lbrace k \in \lbrace 0,\ldots,N-1 \rbrace \mid \norm{(x_{u}(k,x_0) - \xeq, u(k) - \ueq)} \leq \eps \rbrace$ satisfies the inequality $Q_{\eps} \geq N- (\delta+C)/\rho(\eps)$. 
\end{lemma}

We note that this result holds for much more general optimal control problems than specified above. In the setting from above, we have $\ell(\xeq,\ueq) = 0$ and thus we will see the turnpike phenomenon for all solutions $x_{u}(k,x_0)$ for which $J_N(x_0,u)$ is uniformly bounded in $N$.  We note that although in the sequel we do not change the structure of the cost function, we will not have that the cost vanishes in the stochastic counterpart of the turnpike steady state $(\xeq,\ueq)$ that we identify below.

\section{STATIONARY SOLUTIONS AND OVERTAKING OPTIMALITY} \label{sec_stationary}

Our aim is to transfer the results from the deterministic to the stochastic setting. For that, we have to consider how we can receive an equivalent formulation of \eqref{eq_optStSt_det} to characterize the optimal steady state for system \eqref{eq_StochSystem}. 
We note that, for any $U \in \Ltwo{\R^l}$, the condition
\begin{equation}
    X = AX + BU + W
\end{equation}
cannot be satisfied as $U$ may not depend on $W$ according to the underlying filtration. Thus, a random variable, which is constant in time, is not a suitable candidate for the stochastic counterpart of the optimal steady state. 
However, we can keep the distribution of such a process constant, instead,  
and use the following definition of stationary processes. 
\begin{definition} \label{defn_StationaryPair}
    A pair of state and control processes $(\Xstat(\cdot),\Ustat(\cdot))$ is called stationary for system~\eqref{eq_StochSystem} if there is a stationary distribution $\varrho^s$ such that
    \begin{equation*}
        \Xstat(k) \sim \varrho^s \quad \mbox{and} \quad \Xstat(k+1) = A\Xstat(k) + B\Ustat(k) + W(k) 
    \end{equation*}
    for all $k \in \N_0$.
\end{definition}
\begin{Bem}[Stationarity in probability measures]~\\ \label{rem_StatProbMeasure}
The probability measure defining a stationary distribution is called an invariant measure, see \cite{Meyn1993}.
Hence, an alternative approach to the above stationarity concept could be to switch entirely to the set of underlying probability measures and to conduct the analysis there, see \cite{Gros2022}.
However, this approach is limited in that  all information about the single realization paths of the solutions is lost -- although numerical simulations suggest a turnpike phenomenon for the paths, as well, see \cite{Ou2021}. 
\end{Bem}

As shown in \cite{Meyn1989}, every control of the form $U(k) := LX(k)$ with $A + BL$ Schur-stable yields an invariant probability measure and, therefore, such a stationary pair. 
So we need an additional optimality criterion to characterize  optimal stationary pairs along with the optimal $L$. Since we aim at formulating this criterion regarding the stochastic processes, which are not constant over time, we cannot minimize the stage cost \eqref{eq_StageCost} for one single time step. Instead, we must include the entire evolution of $X(\cdot)$ and $U(\cdot)$ in the optimization process. 
For this purpose, let us first derive a reformulation of the cost.
\begin{lemma} \label{lem_modifiedCost}
    Let $P$ be the unique positive semidefinite solution of the discrete-time algebraic Riccati equation 
    \begin{equation} \label{eq_algRiccati}
        P = A^T P A + Q - A^T P B [ R + B^T P B ]^{-1} B^T P A,
    \end{equation}
    and set $\RF := -\left[R + B^T P B\right]^{-1} B^T P A$. Then for every $N \in \N$ the cost \eqref{eq_sOCP} can be rewritten as 
    \begin{equation*}
        \begin{split}
            J_N(X_0,U) =& \sum_{k=0}^{N-1} \Exp{\norm{U(k) - \RF X(k)}_{\widetilde{R}}^2} + \Exp{\norm{X_0}_P^2} \\
            &- \Exp{\norm{X(N)}_P^2} + \sum_{k=0}^{N-1} \Exp{\norm{W(k)}_P^2}
        \end{split}
    \end{equation*}
    with $\widetilde{R} := R+B^TPB$ symmetric and positive definite. 
\end{lemma}

\begin{proof}
    Since $P$ is the constant solution of the Riccati iteration \eqref{eq_RiccatiIteration} with terminal condition $P_N(N)=P$, this lemma is a direct consequence of \cite[Chapter~8, Lemma~6.1]{Aastroem1970} and the stochastic independence of $X(k)$ and $W(k)$ as well as $U(k)$ and $W(k)$.
\end{proof}
Lemma \ref{lem_modifiedCost} reveals that for every control $U(\cdot)$, the cost grows unbounded if we let $N$ go to infinity. That means we cannot directly solve the infinite horizon stochastic OCP. Thus, we use the following concept of \textit{overtaking optimality}, originally introduced by \cite{Atsumi1965,vonWeizsaecker1965}. Moreover, it was already used to analyze turnpike behavior for time-varying discrete-time optimal control problems in deterministic settings, see \cite{Gruene2018a}. 
\begin{definition}
    Let $X_0$ be a fixed initial condition. Then a control $U_1(\cdot)$ overtakes $U_2(\cdot)$ if there is a time $N_0$ such that
    \begin{equation*}
        J_N(X_0,U_1) < J_N(X_0,U_2), \quad \text{for all } N>N_0.
    \end{equation*}
    For a set of controls $\U$, a control $U^*(\cdot)\in\U$ is called overtaking optimal on $\U$
    if it overtakes every other control $U(\cdot) \in \U$. 
\end{definition}

Note that this definition implies the uniqueness of the overtaking optimal control if it exists.
For the continuous-time case, it has been shown in \cite{Leizarowitz1987} that the feedback law obtained from the solution of the algebraic Riccati equation defines an overtaking optimal control. We adapt this proof and show that the claim also holds for the discrete-time case. 
\begin{thm} \label{thm_overtaking optimal}
    For every initial condition $X_0 \sim \mathcal{N}(\mu_0,\Sigma_0)$ there exists a unique overtaking optimal control $U^*(\cdot)$ on the set
    \begin{equation*}
        \begin{split}
            \U = \lbrace U(\cdot) \mid~ &\exists \alpha>0:~\Exp{\Vert X_{U}(k,X_0) \Vert^2} \leq \alpha \mbox{ and} \\
            &U(k) \in \Ltwo{\R^l} \mbox{ for all } 
            k\in\N_0 \rbrace
        \end{split}
    \end{equation*}
    given by the feedback law 
    \begin{equation} \label{eq: optimal_U}
        U^*(k) = \RF X^*(k), \quad X^*(\cdot) := X_{U^*}(\cdot,X_0)
    \end{equation}
    where $\RF := -\left[R + B^T P B\right]^{-1} B^T P A$ and $P$ is the positive semidefinite symmetric solution of the discrete-time algebraic Riccati equation \eqref{eq_algRiccati}.
    Further, the following stronger version of overtaking optimality holds: For every $U(\cdot) \in \U$ with $U(\cdot) \neq U^*(\cdot)$ there is a time $N_0$ and an $\varepsilon > 0$ such that 
    \begin{equation*}
        J_N(X_0, U^*) + \varepsilon < J_N(X_0, U)
    \end{equation*}
    for all $N \geq N_0$.
\end{thm}

\begin{proof}
    First we have to verify that $U^*(\cdot)=KX^*(\cdot) \in \U$. Since $U^*(\cdot)$ only depends on the current state, it is clear that the filtration condition is fulfilled and thus $U^*(k) \in \Ltwo{\R^l}$ holds for all $k \in \N_0$. 
    So it remains to show that $\E[\Vert X^*(k) \Vert^2]$ is bounded for all $k \in \N_0$. We can calculate the distribution of $X^*(k) \sim \Normal(\mu(k),\Sigma(k))$ directly by the equations
    \begin{equation} \label{eq_SystemDistribution}
        \begin{split}
            \mu(k+1) &= A_K \mu(k) \\
            \Sigma(k+1) &= A_K \Sigma(k) A_K^T + \Sigma_W 
        \end{split} 
    \end{equation}
    with the initial condition $(\mu(0),\Sigma(0)) = (\mu_0,\Sigma_0)$ and the Schur-stable matrix $A_{\RF} := A + B\RF$. Therefore, we know that $\mu(k)$ and $\Sigma(k)$ converge because of the Schur-stability of $A_{\RF}$ and, thus, we get that
    \begin{equation*}
        \Exp{\norm{X(k)}^2} = \sum_{i=0}^{n} \Exp{X_i(k)^2} = \mu(k)^T \mu(k) + Tr(\Sigma(k))
    \end{equation*}
    converges, too, which especially implies the boundedness of $\E[\Vert X^*(k) \Vert^2]$. 
    
    Now, since we know that $U^*(\cdot) \in \U$, we can write every 
    \begin{equation*}
        U(k) = \RF X(k) + V(k), \quad \mbox{for all} ~ k \in \N_0 
    \end{equation*}
    with $V(\cdot) \in \U$. 
    If we use Lemma \ref{lem_modifiedCost} to calculate the cost of $U(\cdot)$ we get 
    \begin{equation*}
        \begin{split}
            J_N(X_0,U) =& \sum_{k=0}^{N-1} \Exp{\norm{V(k)}_{\widetilde{R}}^2} + \Exp{\norm{X_0}_P^2} \\
            &- \Exp{\norm{X(N)}_P^2} + \sum_{k=0}^{N-1} \Exp{\norm{W(k)}_P^2}
        \end{split}
    \end{equation*}
    and, thus,
    \begin{equation*}
        \begin{split}
            &J_N(X_0,U) - J_N(X_0,U^*) \\
            &= \sum_{k=0}^{N-1} \Exp{\norm{V(k)}_{\widetilde{R}}^2} - \Exp{\norm{X(N)}_P^2} + \Exp{\norm{X^*(N)}_P^2}.
        \end{split}
    \end{equation*}
    We consider two cases to prove that $U^*(\cdot)$ is overtaking optimal. First assume $\lim_{N \rightarrow \infty} \sum_{k=0}^{N-1} \E [ \Vert V(k) \Vert_{\widetilde{R}}^2 ] = \infty$.
    Then we get by the boundedness of $\E[\norm{X(T)}^2]$ and the nonnegativity of $\E[\norm{X^*(T)}_P^2 ]$ that there exists a time $N_0 > 0$ and a constant $\alpha_0 > 0$ such that 
    \begin{equation*}
        J_N(X_0,U) - J_N(X_0,U^*) \geq \sum_{k=0}^{N-1} \Exp{\norm{V(k)}_{\widetilde{R}}^2} - \alpha_0 > \eps
    \end{equation*}
    for all $N > N_0$, i.e.,\ $U^*$ is overtaking optimal. 
    
    Let us now consider the second possibility 
    \begin{equation} \label{eq_Vsummable}
        0 \leq \lim_{N \rightarrow \infty} \sum_{k=0}^{N-1} \Exp{\norm{V(k)}_{\widetilde{R}}^2} =: \gamma < \infty.
    \end{equation}
     First, we assume $\gamma = 0$. Then we can conclude that $\E[\Vert V(k) \Vert_{\tilde{R}}^2] = 0$ for all $k \in \N_0$ and thus $V(\cdot) = 0$ almost surely, which implies $U(\cdot) = U^*(\cdot)$.
    Therefore, we can w.l.o.g. assume that $\gamma > 0$, which implies that there exists a constant $\tilde{\eps} > 0$ and a time $N_1 > 0$ such that 
    \begin{equation}
        \sum_{k=0}^{N-1} \Exp{\norm{V(k)}_{\widetilde{R}}^2} > \tilde{\eps} \quad \text{ for all } N \geq N_1.
    \end{equation}
    Next, we show that 
    \begin{equation} \label{eq_SecondMoment}
        \lim_{N \rightarrow \infty} \left[\Exp{\norm{X(N)}_P^2} - \Exp{\norm{X^*(N)}_P^2}\right] = 0
    \end{equation}
    for which by the equality $x^TPx = tr(Pxx^T)$ it is sufficient to prove that 
    \begin{equation*}
        \lim_{N \rightarrow \infty} \left[\Exp{X(N) X(N)^T} - \Exp{X^*(N) X^*(N)^T}\right] = 0.
    \end{equation*}
    The equations satisfied by $X^*(\cdot)$ and $X(\cdot)$ are
    \begin{align*}
        X^*(k+1) &= A_{\RF} X^*(k) + W(k) \\
        X(k+1) &= A_{\RF} X(k) + BV(k) + W(k)
    \end{align*}
    and thus, the solution $X(\cdot)$ is given by
    \begin{equation} \label{eq_solX}
        \begin{split}
            X(k) &= X^*(k) + \sum_{j=0}^{k-1} A_{\RF}^j B V(k-1-j).
        \end{split}
    \end{equation}
    From equation \eqref{eq_solX} we get
    \begin{equation*}
        \begin{split}
            &\Exp{X(k)X(k)^T} =  \Exp{X^*(k)X^*(k)^T} \\
            &+ \E \Big[ X^*(k) \tilde{V}(k)^T \Big] + \E \Big[ \tilde{V}(k) X^*(k)^T \Big] + \E\Big[ \tilde{V}(k) \tilde{V}(k)^T \Big]
        \end{split}
    \end{equation*}
    with $\tilde{V}(k) := \sum_{j=0}^{k-1} A_{\RF}^j B V(k-1-j)$.
    It follows from the Cauchy-Schwarz inequality that there is a constant $\kappa > 0$ such that any two random variables $Z_1$ and $Z_2$ satisfy
    \begin{equation} \label{eq_CauchySchwarz}
        \norm{\Exp{Z_1 Z_2^T}} \leq \kappa \left(\Exp{\norm{Z_1 }^2}\right)^{1/2} \left(\Exp{\norm{Z_2}^2}\right)^{1/2}.
    \end{equation}
    By inequality \eqref{eq_CauchySchwarz}, it will be sufficient to prove that
    $\E[ \Vert \tilde{V}(k) \Vert^2 ] \rightarrow 0$ as $k \rightarrow \infty$ in order to establish $\E[X(k)X(k)^T - X^*(k)X^*(k)^T] \rightarrow 0$. 
    It holds that
    \begin{equation*}
        \Exp{\big\Vert \tilde{V}(k) \big\Vert^2} \leq \left\Vert B \right\Vert^2 \sum_{j=0}^{k-1}\big\Vert A_{\RF}^j\big\Vert^2 ~\Exp{\big\Vert V(k-1-j) \big\Vert^2}
    \end{equation*}
    which converges to zero for $k \rightarrow \infty$ because $\|A_K^j\|^2$ and $\E[\|V(j)\|^2]$ are summable by Schur-stability of $A_K$ and equation \eqref{eq_Vsummable}, respectively.
    This proves \eqref{eq_SecondMoment}, and thus there exist a time $N_0 \geq N_1$ such that
    \begin{equation*}
        \begin{split}
            &J_N(X_0,U) - J_N(X_0,U^*) \\
            &= \sum_{k=0}^{N-1} \Exp{\norm{V(k)}_{\widetilde{R}}^2} + \Exp{\norm{X^*(N)}_P^2}  - \Exp{\norm{X(N)}_P^2} \\
            &> \tilde{\eps} - \dfrac{\tilde{\eps}}{2} = \dfrac{\tilde{\eps}}{2} := \eps > 0
        \end{split}
    \end{equation*}
    for all $N > N_0$, which concludes the proof. 
\end{proof}
Note that the restriction to the set $\U$ in Theorem \ref{thm_overtaking optimal} is not limiting, since we are only interested in optimal stationary pairs and every control defining a stationary pair must lie in $\U$ to provide an invariant measure.
Next, we show that with a proper choice of $X_0$ the overtaking optimal control process from Theorem \ref{thm_overtaking optimal} defines a stationary pair in the sense of Definition \ref{defn_StationaryPair}.
\begin{lemma} \label{lem_optimalStPair}
    For system \eqref{eq_StochSystem}, there exists a stationary pair with control
    $\Ustatopt(\cdot) = \RF \Xstatopt(\cdot)$, where $\RF$ is the feedback law generating the overtaking optimal control from Theorem \ref{thm_overtaking optimal}, together with the state process 
    \begin{equation*}
        \Xstatopt(k+1) = (A+B \RF) \Xstatopt(k) + W(k)
    \end{equation*}
    and initial distribution $\Xstatopt(0) \sim \mathcal{N}(\mu^s_*, \Sigma^s_*)$. Here $\mu^s_* = 0$ and $\Sigma^s_*$ is the solution of the Lyapunov equation
    \begin{equation*} 
        \Sigma^s_* = (A + B\RF) \Sigma^s_* (A + B\RF)^T + \Sigma_W.
    \end{equation*}
\end{lemma}
\begin{proof}
    Since $A+B \RF$ is Schur-stable, by \cite{Meyn1989} the stationary state process $\Xstatopt(\cdot)$ exists.
    From equation \eqref{eq_SystemDistribution} it is clear that the corresponding invariant measure is given by $\mathcal{N}(\mu^s_*, \Sigma^s_*)$ where $\mu^s_*$ and $\Sigma^s_*$ must satisfy the equations 
    \begin{equation*}
        \mu^s_* = A_{\RF} \mu^s_*, \quad \Sigma^s_* = A_{\RF} \Sigma^s_* A_{\RF}^T + \Sigma_W. \vspace{-0.6cm}
    \end{equation*}
\end{proof}

\section{STOCHASTIC DISSIPATIVITY AND TURNPIKE} \label{sec_DissiAndTurnppike}
As we have characterized the optimal stationary pair for the considered stochastic OCP, we now show that this pair satisfies a stochastic extension of strict dissipativity. 
In the resulting dissipativity inequality, the distance measure on the right-hand side depends on the exact realization of the random variables rather than only on their distributions, cf. \eqref{eq:sDI}.
That means our distance measure is not a dissimilarity measure in the sense of \cite{Gros2022}, and, thus, our concept of stochastic dissipativity differs from that one presented there.
Specifically, we construct a modified cost function and a corresponding storage function in several steps in a series of lemmas, starting with the stage cost from Lemma \ref{lem_modifiedCost}.
\begin{lemma} \label{lem_DissiStrategy}
    For every control $U(\cdot)$ and corresponding state $X(\cdot) = X_{U}(\cdot,X_0)$, the equality
    \begin{equation*}
        \begin{split}
            &\hat{\ell}(X(k),U(k)) := \ell(X(k),U(k)) - \ell(\Xstatopt(k),\Ustatopt(k)) \\ 
            &~~+ \hat{\lambda}(X(k)) - \hat{\lambda}(X(k+1)) = \Exp{\norm{U(k) - \RF X(k)}_{\widetilde{R}}^2}
        \end{split}
    \end{equation*}
    holds with $\hat{\lambda}(X(k)) = -\E[\Vert X(k) \Vert_P^2]$.
\end{lemma}
\begin{proof}
    Using Lemma \ref{lem_modifiedCost} with $N=1$ we get 
    \begin{equation*}
        \begin{split}
            \ell(X(k),U(k)) =& 
            \Exp{\norm{U(k) - \RF X(k)}_{\widetilde{R}}^2} + \Exp{\norm{X(k)}_P^2}  \\
            &- \Exp{\norm{X(k+1)}_P^2} + \Exp{\norm{W(k)}_P^2}
        \end{split}
    \end{equation*}
    and $\ell(\Xstatopt(k),\Ustatopt(k)) = \E[\norm{W(k)}_P^2]$. 
    This results in
    \begin{equation*}
        \begin{split}
            &\hat{\ell}(X(k),U(k)) = \ell(X(k),U(k)) \\
            &~-\Exp{\norm{W(k)}_P^2} - \Exp{\norm{X(k)}_P^2} + \Exp{\norm{X(k+1)}_P^2} \\
            &= \Exp{\norm{U(k) - \RF X(k)}_{\widetilde{R}}^2}
        \end{split}
    \end{equation*}
    which proves the claim.
\end{proof}
If we see the control $U(\cdot)$ as a kind of strategy $\pi(\cdot)$ of the form $U(k) = \pi(X(k))$ then Lemma \ref{lem_DissiStrategy} could already be interpreted as a dissipativity equation with respect to the strategy. 
However, since we are interested in deriving a stochastic version of $(x,u)$-dissipativity, we further modify this equation, leading to the following lemma.
For simplicity, we denote in the following the difference between a stochastic state or control process and the optimal stationary process by $\widetilde{X}(\cdot):=X(\cdot) - \Xstatopt(\cdot)$, respectively, by $\widetilde{U}(\cdot):=U(\cdot) - \Ustatopt(\cdot)$.
\begin{lemma} \label{lem_DissiMeanSquare}
    For every control $U(\cdot)$ it holds that 
    \begin{equation*}
        \begin{split}
            \bar{\ell}(X(k),U(k)) :=& \ell(X(k),U(k)) - \ell(\Xstatopt(k),\Ustatopt(k)) \\
            &+ \bar{\lambda}(k,X(k)) - \bar{\lambda}(k+1,X(k+1)) \\
            = \E\Big[\Vert X(k)&-\Xstatopt(k)\Vert_Q^2  + \Vert U(k) - \Ustatopt(k) \Vert_R^2 \Big]
        \end{split}
    \end{equation*}
    for every $k \in \N_0$. Here,  $\bar{\lambda}$ is given by 
    \begin{equation*}
        \bar{\lambda}(k,X(k)) := -\Exp{\norm{X(k)}_P^2 - \norm{X(k) - \Xstatopt(k)}_P^2}. 
    \end{equation*}
    
\end{lemma}

\begin{proof}
    First, we observe that
    \begin{equation*}
        \begin{split}
            &\Exp{\norm{U(k) - \RF X(k)}_{\widetilde{R}}^2} \\
            &= \Exp{\norm{\left( U(k) - \Ustatopt(k) \right) - \RF \left(X(k) - \Xstatopt(k) \right)}_{\widetilde{R}}^2}
        \end{split}
    \end{equation*}
    holds with $\Ustatopt(k) = \RF \Xstatopt(k)$. Thus, we obtain $\hat{\ell}(X(k),U(k)) = \hat{\ell}(\widetilde{X}(k),\widetilde{U}(k))$ by Lemma \ref{lem_DissiStrategy}
    and therefore
    \begin{equation*}
        \begin{split}
            &\Exp{\Vert \widetilde{X}(k) \Vert_Q^2 + \Vert \widetilde{U}(k) \Vert_R^2} = \ell(\widetilde{X}(k),\widetilde{U}(k)) \\
            &= \hat{\ell}(X(k),U(k)) + \ell(\Xstatopt(k),\Ustatopt(k))\\
            &~~~ - \hat{\lambda}(\widetilde{X}(k)) + \hat{\lambda}\Big( A\widetilde{X}(k) + B \widetilde{U}(k) + W(k) \Big).
        \end{split}
    \end{equation*}
    Further, because $(X(k),U(k))$ as well as $(\Ustatopt(k),\Xstatopt(k))$ are stochastically independent of $W(k)$ we get
    \begin{equation*}
        \begin{split}
            &\hat{\lambda}\Big( A \widetilde{X}(k) + B \widetilde{U}(k) + W(k) \Big) = - \Exp{\norm{W(k)}_P^2} \\
            &~~~-\Exp{\norm{AX(k) + BU(k) - A\Xstatopt(k) - B\Ustatopt(k)}_P^2} \\ 
            &= - \Exp{\Vert \widetilde{X}(k+1) \Vert_P^2} - \ell(\Xstatopt(k),\Ustatopt(k)).
        \end{split}
    \end{equation*}
    Putting all these equations together, we finally get
    \begin{equation*}
        \begin{split}
            &\Exp{\norm{X(k)-\Xstatopt(k)}_Q^2 + \norm{U(k) - \Ustatopt(k)}_R^2} \\
            &= \hat{\ell}(X(k),U(k)) + \Exp{\norm{X(k) - \Xstatopt(k)}_P^2} \\
            &~ - \Exp{\norm{X(k+1) - \Xstatopt(k+1)}_P^2} = \bar{\ell}(X(k),U(k)).
        \end{split}
    \end{equation*}
     \begin{equation*}
         \vspace{-1.cm}
     \end{equation*}
\end{proof}

With a final modification of the stage cost we can replace the positive semidefinite matrix $Q$ in the weight of the state with a positive definite matrix. This leads to the next lemma.

\begin{lemma} \label{lem_DissiPosDef}
    There exists symmetric and positive definite matrices $S \in \R^{n \times n}$, $H \in \R^{(n+l) \times (n+l)}$, such that 
    for every control $U(\cdot)$ it holds that 
    \begin{equation*}
        \begin{split}
            \tilde{\ell}(X(k),U(k)) :=& \ell(X(k),U(k)) - \ell(\Xstatopt(k),\Ustatopt(k)) \\
            &+  \tilde{\lambda}(k,X(k)) - \tilde{\lambda}(k+1,X(k+1))\\
            =& \Exp{\norm{(X(k)-\Xstatopt(k),U(k) - \Ustatopt(k))}_H^2}
        \end{split}
    \end{equation*}
    with $\tilde{\lambda}(k,X(k)) := \bar{\lambda}(k,X(k)) + \Exp{\Vert X(k)-\Xstatopt(k) \Vert_S^2}$.
\end{lemma} 

\begin{proof}
    Since $(A,Q^{1/2})$ is detectable, we know by \cite[Lemma 5.4]{Gruene2018} that there is $\tilde{S} \in \R^{n \times n}$ symmetric and positive definite satisfying the matrix inequality $Q + \tilde{S} - A^T\tilde{S}A > 0$.
    For a given $\gamma \in (0,1]$, set $\widetilde{S}_{\gamma} := \gamma \widetilde{S}$ and $ Q_{\gamma} := Q + \widetilde{S}_{\gamma} - A^T \widetilde{S}_{\gamma}A$.
    Then following the calculation of \cite[Lemma 4.1]{Gruene2018} we get 
    \begin{equation*}
        \begin{split}
            &\ell(\widetilde{X}(k),\tilde{U}(k)) + \Exp{\Vert \widetilde{X}(k) \Vert_{\widetilde{S}_{\gamma}}^2} - \Exp{\Vert \widetilde{X}(k+1) \Vert_{\widetilde{S}_{\gamma}}^2} \\
            &= \bar{\ell}(X(k),U(k)) + \Exp{\Vert \widetilde{X}(k) \Vert_{\widetilde{S}_{\gamma}}^2} \\
            &\quad - \Exp{\Vert A\widetilde{X}(k)+B\widetilde{U}(k)\Vert_{\widetilde{S}_{\gamma}}^2} = \Exp{\Vert (\widetilde{X}(k),\widetilde{U}(k)) \Vert_H^2}
        \end{split}
    \end{equation*}
    with 
    \begin{equation*}
        H := \dfrac{1}{2}\left( 
        \begin{array}{cc}
            2Q_{\gamma} & \gamma E \\
             \gamma E & 2 R_{\gamma}
        \end{array}
        \right),
    \end{equation*}
    $R_{\gamma} := R-B^T\widetilde{S}_{\gamma}B$ and $E := -A^T\widetilde{S}B - B^T\widetilde{S}A$.
    Using the Schur complement we can show, that $H$ is positive definite for a sufficient small $\tilde{\gamma} \in (0,1]$, see \cite{Gruene2018}, which proves the lemma with $S = \widetilde{S}_{\tilde{\gamma}} > 0$.
\end{proof}
Note that,
\begin{equation*}
    \begin{split}
        &\tilde{\lambda}(k,X(k)) = \Exp{\norm{X(k) - \Xstatopt(k)}_{P+S}^2 - \norm{X(k)}_P^2 } \\
        =& \Exp{\norm{X(k)}_S^2 - 2X(k)^T(P+S)\Xstatopt(k) + \norm{\Xstatopt(k)}_{P+S}^2}
    \end{split}
\end{equation*}
is bounded from below since $S>0$ and $P \geq 0$ and thus, Lemma \ref{lem_DissiPosDef} delivers a dissipativity equation equivalently to Definition \ref{defn_DissiDeterm}. 
Therefore, we can show a turnpike property for the stochastic case analogous to Lemma \ref{lem_TurnpikeDeterm}. 

\begin{thm} \label{thm_StochTurnpike}
   For each $X_0 \sim \Normal(\mu_0,\Sigma_0)$, there exists a constant $C \in \R$ such that for each $\delta > 0$, each control process $U(\cdot)$ satisfying $J_N(X_0,U) \leq \delta + \sum_{k=0}^{N-1} \ell(\Xstatopt(k),\Ustatopt(k))$ and each $\eps > 0$ the value
    \begin{equation*}
        Q_{\eps} := \# \Big\lbrace k \in \lbrace 0, \ldots, N-1 \rbrace \mid\E \big[ \Vert(\widetilde{X}(k),\widetilde{U}(k))\Vert_H^2 \big] \leq \eps \Big\rbrace
    \end{equation*}
    satisfies the inequality $Q_{\eps} \geq N-(\delta+C)/\eps$ for all $N \in \N$.
\end{thm}

\begin{proof}
    The proof follows the same arguments as \cite[Theorem 5.3]{Gruene2013}.
    Set $C := \tilde{\lambda}(0,X_0) - M$ where $M \in \R$ is a lower bound on $\tilde{\lambda}$. Then for $\tilde{J}_N(X_0,U) := \sum_{k=0}^{N-1} \tilde{\ell}(X(k),U(k))$ we get 
    \begin{equation} \label{eq_turnpikeContr}
        \begin{split}
            \tilde{J}_N(X_0,U) =& J_N(X_0,U) - \sum_{k=0}^{N-1} \ell(\Xstatopt(k),\Ustatopt(k)) \\
            &+ \tilde{\lambda}(0,X(0)) - \tilde{\lambda}(N,X(N)) \leq \delta + C.
        \end{split}
    \end{equation}
    Now assume that $Q_{\eps} < N-(\delta+C)/\eps$. This means there is a set $\mathcal{\Bar{M}} \subset \lbrace 0,\ldots,N-1 \rbrace$ of $N-Q_{\eps} > (\delta+C)/\eps$ time instants such that $\E [ \Vert (\widetilde{X}(k),\widetilde{U}(k)) \Vert_H^2 ] \geq \eps$ for all $k \in \mathcal{\Bar{M}}$. Using Lemma \ref{lem_DissiPosDef}, this implies $\tilde{J}_N(X_0,U) = \sum_{k=0}^{N-1} \E\big[\Vert (\widetilde{X}(k),\widetilde{U}(k)) \Vert_H^2\big] > \delta + C$, which contradicts \eqref{eq_turnpikeContr} and, thus, proves the theorem.
\end{proof}

It is not immediately obvious what Theorem \ref{thm_StochTurnpike} means for the single realization paths of $X(\cdot)$ and $U(\cdot)$. What we cannot conclude is an almost sure turnpike property of the paths. However, what we can conclude is a turnpike property in probability. This means the probability that a realization is not near $(\Xstatopt(\cdot), \Ustatopt(\cdot))$ is small in an appropriate probabilistic sense, except for a number of time instances independent of $N$. The following theorem formalizes this result. 
\begin{thm} \label{thm_TurnpikeProb}
    For each $X_0 \sim \Normal(\mu_0,\Sigma_0)$, there exists a constant $C \in \R$ such that for each $\delta > 0$, each control process $U(\cdot)$ satisfying $J_N(X_0,U) \leq \delta + \sum_{k=0}^{N-1} \ell(\Xstatopt(k),\Ustatopt(k))$ and each $\epsilon, \eta > 0$ the value
    \begin{equation*}
        P_{\epsilon, \eta} := \# \Big\lbrace k \in \mathcal{M}_N \mid\Prob\Big( \Vert (\widetilde{X}(k),\widetilde{U}(k)) \Vert_H^2 \geq \epsilon \Big)\leq \eta \Big\rbrace
    \end{equation*}
    satisfies $P_{\epsilon,\eta} \geq N-(\delta+C)/(\epsilon \eta)$ for all $N \in \N$, where $\mathcal{M}_N := \lbrace 0, \ldots, N-1 \rbrace$.
\end{thm}
\begin{proof}
    Using the Markov inequality, we get 
    \begin{equation} \label{eq_MarkovIneq}
        \Prob\big(\Vert (\widetilde{X}(k),\widetilde{U}(k)) \Vert_H^2 \geq \epsilon \big) \leq \dfrac{1}{\epsilon}\E\big[\Vert (\widetilde{X}(k),\widetilde{U}(k)) \Vert_H^2 \big].
    \end{equation}
    Further, by Theorem \ref{thm_StochTurnpike}, we know that there are at least $N-(\delta+C)/(\epsilon \eta)$ time instants for which $\E [ \Vert (\widetilde{X}(k),\widetilde{U}(k)) \Vert_H^2 ] \leq \epsilon \eta$.
    Using equation \eqref{eq_MarkovIneq}, this gives $\Prob\left(\Vert (\widetilde{X}(k),\widetilde{U}(k)) \Vert_H^2 \geq \epsilon \right) \leq \eta$ for all these time instants and, thus, proves the claim.
\end{proof}

Note that in contrast to the classical turnpike approaches, we have to fix two constants in Theorem \ref{thm_TurnpikeProb} where $\epsilon$ defines how close we want to be to the realization of the stationary pair and $\eta$ defines the desired probability bound for this.

\section{NUMERICAL EXAMPLE} \label{sec_example}
Consider the one-dimensional stochastic OCP 
\begin{equation} \label{eq_exampleOCP}
    \begin{split}
        \min_{U(\cdot)} \sum_{k=0}^{N-1} &\Exp{X(k)^2 + 5U(k)^2} \\
        \mbox{s.t.}~X(k+1) &= 1.2X(k) + U(k) + W(k)
    \end{split}
\end{equation}
with initial condition $X(0) = X_0 \sim \Normal(3,1.5)$ and $W(k) \sim \Normal(0,10)$ for all $k=0,\ldots,N-1$. Then the corresponding stationary pair from Lemma \ref{lem_optimalStPair} is defined by the feedback $K \approx -0.558$ and the initial distribution $\Xstatopt(0) \sim \Normal(0,17)$. 
To observe the pathwise turnpike property from Theorem \ref{thm_TurnpikeProb}, we fix a realization of $W(\cdot)$ given by $w = \lbrace w_k \rbrace$ and simulate the states and controls according to the equation \eqref{eq_exampleOCP} and the optimal stationary pair with $W(k) = w_k$ and random initial values following the desired distributions.
Figure \ref{fig_pathwiseTurnpike} shows the noise-sequence $w$ and the resulting turnpike behavior of the realization. 
Further, Table~\ref{table:1} shows the number of points $P^{w}_{0.25} := \# \lbrace k \in \mathcal{M}_N \mid \Vert (\widetilde{X}(k),\widetilde{U}(k)) \Vert_H^2 \geq 0.25 \rbrace$ outside the $\eps=0.25$ neighborhood of the stationary pair for the paths from Figure~\ref{fig_pathwiseTurnpike} for the fixed realization $w$ of the noise. It can be observed that the value $P^{w}_{0.25}$ is uniformly bounded by $12$, which further illustrates the pathwise turnpike property of this realization.

\begin{table}[ht]
\vspace{-0.05cm}
    \centering
    \begin{tabular}{ |c||c|c|c|c|c|c|c|c|c|c|c| } 
        \hline
        $N$ & 10 & 15 & 20 & 25 & 30 & 35 & 40 & 45 & 50 \\
        \hline
        $P^{w}_{0.25}$ & 10 & 9 & 11 & 12 & 9 & 10 & 12 & 10 & 12 \\
        \hline
    \end{tabular}
    \caption{Number of points $P^{w}_{0.25}$ for different horizons $N$.} \label{table:1}
    \vspace{-0.95cm}
\end{table}

\begin{figure}[ht] 
    \centering
    \includegraphics[width=8.3cm,height=4.2cm]{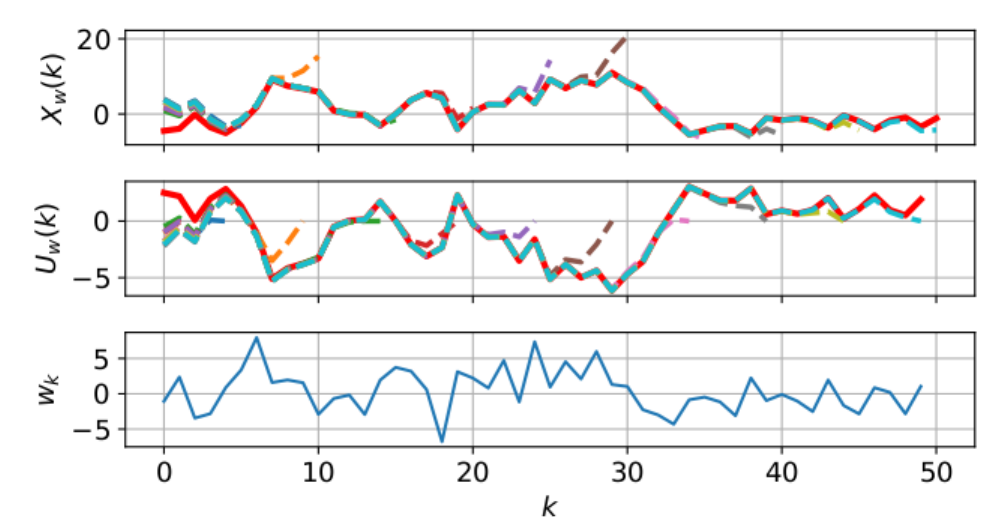}
    \caption{Fixed realization $w$ of the noise and corresponding realizations of the optimal solutions $(X_N^*(\cdot),U_N^*(\cdot))$ from \eqref{eq_exampleOCP} on different horizons $N$ (dashed) and the optimal stationary pair $(\Xstatopt(\cdot),\Ustatopt(\cdot))$ (solid red).}
    \label{fig_pathwiseTurnpike}
    \vspace{-0.5cm}
\end{figure}

\section{SUMMARY} \label{sec_summary}
This paper presented pathwise turnpike results in probability based on strict dissipativity for discrete-time stochastic linear-quadratic optimal control problems. Compared to the deterministic setting, a stationary stochastic process replaces the optimal steady state. We have illustrated the theoretical results considering a one-dimensional example.

{\bibliographystyle{abbrv} 
  \bibliography{turnpike_stochastic_LQR} 
}

\end{document}